\documentclass[10pt,reqno]{amsart}
\usepackage{bbm}
\usepackage{mathrsfs}
\usepackage{amsfonts} 
\usepackage[dvipsnames,usenames]{color}
\usepackage{graphicx,latexsym,bm,amsmath,amssymb,verbatim,multicol,lscape}
\newtheorem{thm}{Theorem} [section]

\newtheorem{lem}[thm]{Lemma}

\theoremstyle{definition}

\theoremstyle{remark}

\newtheorem{con}[thm]{Conjecture}
\numberwithin{equation}{section}

\begin{document}
\title{A certain reciprocal power sum is never an integer}
\begin{abstract}
By $(\mathbb{Z}^+)^{\infty}$ we denote the set of all the
infinite sequences $\mathcal{S}=\{s_i\}_{i=1}^{\infty}$ of positive
integers (note that all the $s_i$ are not necessarily distinct and not
necessarily monotonic). Let $f(x)$ be a polynomial of nonnegative
integer coefficients. For any integer $n\ge 1$, one lets
$\mathcal{S}_n:=\{s_1, ..., s_n\}$ and
$H_f(\mathcal{S}_n):=\sum_{k=1}^{n}\frac{1}{f(k)^{s_{k}}}$.
When $f(x)$ is linear, it is proved in [Y.L. Feng, S.F. Hong,
X. Jiang and Q.Y. Yin, A generalization of a theorem of Nagell, 
Acta Math. Hungari, to appear] that for any infinite
sequence $\mathcal{S}$ of positive integers, $H_f(\mathcal{S}_n)$
is never an integer if $n\ge 2$. Now let deg$f(x)\ge 2$. 
Clearly, $0<H_f(\mathcal{S}_n)<\zeta(2)<2$. But it is not 
clear whether the reciprocal power sum $H_f(\mathcal{S}_n)$
can take 1 as its value. In this paper, with the help of a result
of Erd\H{o}s, we use the analytic and $p$-adic method to
show that for any infinite sequence $\mathcal{S}$ of
positive integers and any positive integer $n\ge 2$,
$H_f(\mathcal{S}_n)$ is never equal to 1.
Furthermore, we use a result of Kakeya to show that
if $\frac{1}{f(k)}\le\sum_{i=1}^\infty\frac{1}{f(k+i)}$
holds for all positive integers $k$, then the union set 
$\bigcup\limits_{\mathcal{S}\in (\mathbb{Z}^+)^{\infty}}
\{ H_f(\mathcal{S}_n) | n\in \mathbb{Z}^+ \}$ is dense
in the interval $(0,\alpha_f)$ with
$\alpha_f:=\sum_{k=1}^{\infty}\frac{1}{f(k)}$.
It is well known that $\alpha_f=
\frac{1}{2}\big(\pi \frac{e^{2\pi}+1}{e^{2\pi}-1}-1\big)\approx 1.076674$
when $f(x)=x^2+1$. Our dense result infers that when
$f(x)=x^2+1$, for any sufficiently small
$\varepsilon >0$, there are positive integers $n_1$ and
$n_2$ and infinite sequences $\mathcal{S}^{(1)}$ and
$\mathcal{S}^{(2)}$ of positive integers such that
$1-\varepsilon<H_f(\mathcal{S}^{(1)}_{n_1})<1$ and
$1<H_f(\mathcal{S}^{(2)}_{n_2})<1+\varepsilon$.
\end{abstract}
\author[J.Y. Zhao]{Junyong Zhao}
\address{Mathematical College, Sichuan University, Chengdu 610064, P.R. China}
\email{zhjy626@163.com}
\author[S.F. Hong]{Shaofang Hong$^*$}
\address{Mathematical College, Sichuan University, Chengdu 610064, P.R. China}
\email{sfhong@scu.edu.cn; s-f.hong@tom.com; hongsf02@yahoo.com}
\author[X. Jiang]{Xiao Jiang}
\address{Mathematical College, Sichuan University, Chengdu 610064, P.R. China}
\email{422040631@qq.com}
\thanks{$^*$S.F. Hong is the corresponding author and was supported
partially by National Science Foundation of China Grant \#11771304
and by the Fundamental Research Funds for the Central Universities.}
\maketitle

\section{Introduction}
Let $\mathbb{Z}$, $\mathbb{Z}^+$ and $\mathbb{Q}$ be the set
of integers, the set of positive integers and the set of
rational numbers, respectively. Let $n\in \mathbb{Z}^+$.
In 1915, Theisinger \cite{[T]} showed that the $n$-th harmonic
sum $1+\frac{1}{2}+...+\frac{1}{n}$ is never an integer if $n>1$.
In 1923, Nagell \cite{[N]} extended Theisinger's result by showing
that if $a$ and $b$ are positive integers and $n\ge 2$,
then the reciprocal sum $\sum_{i=0}^{n-1}\frac{1}{a+bi}$
is never an integer. Erd\H{o}s and Niven \cite{[EN]} generalized
Nagell's result by considering the integrality of
the elementary symmetric functions of
$\frac{1}{a}, \frac{1}{a+b}, ..., \frac{1}{a+(n-1)b}$.
In the recent years, Erd\H{o}s and Niven's result \cite{[EN]}
was extended to the general polynomial sequence,
see \cite{[CT]}, \cite{[HW]}, \cite{[LHQW]} and \cite{[WH]}.
Another interesting and related topic is presented in \cite{[YHYQ]}.

By $(\mathbb{Z}^+)^{\infty}$ we denote the set of all the
infinite sequence $\{s_i\}_{i=1}^{\infty}$ of positive integers
(note that all the $s_i$ are not necessarily distinct and not
necessarily monotonic). For any given
$\mathcal{S}=\{s_i\}_{i=1}^{\infty}\in (\mathbb{Z}^+)^{\infty}$,
we let $\mathcal{S}_n:=\{s_1, ..., s_n\}.$
Associated to the infinite sequence $\mathcal{S}$ of positive
integers and a polynomial $f(x)$ of nonnegative integer
coefficients, one can form an infinite sequence
$\{H_f(\mathcal{S}_n)\}_{n=1}^{\infty}$ of positive
rational fractions with $H_f(\mathcal{S}_n)$ being defined as follows:
$$H_f(\mathcal{S}_n):=\sum\limits_{k=1}^{n}\frac{1}{f(k)^{s_{k}}}.$$
Very recently, Feng, Hong, Jiang and Yin \cite{[FHJY]} showed
that when $f(x)$ is linear, the reciprocal power sum
$H_f(\mathcal{S}_n)$ is never an integer if $n\ge 2$.
It is natural to ask the following interesting question:
Is the similar result still true when $f(x)$ is of degree
at least two and nonnegative integer coefficients?

In this paper, our main goal is to study this question.
In fact, by using the analytic and $p$-adic method and with
the help of Erd\H{o}s theorem \cite{[E]} on the distribution in
the arithmetic progression $\{4n+1\}_{n=1}^{\infty}$, we will
show the following result that is the first main result of this paper.

\begin{thm}\label{thm1}
Let $f(x)$ be a polynomial of nonnegative integer coefficients
and of degree at least two. Then for any infinite sequence
$\mathcal{S}$ of positive integers and for any positive integer
$n\ge 2$, the reciprocal power sum $H_f(\mathcal{S}_n)$
is never an integer.
\end{thm}
\noindent Clearly, Theorem \ref{thm1} gives an affirmative
answer to the above mentioned question.

Associated to any given infinite sequence $\mathcal{S}$ of
positive integers, we let
$$H_f(\mathcal{S}):=\{H_f(\mathcal{S}_n)|n\in\mathbb{Z}^+\}$$
and
$$\alpha_f(\mathcal{S}):=\sum_{k=1}^{\infty}\frac{1}{f(k)^{s_k}}.$$
Put
$$\alpha_f:=\sum_{k=1}^{\infty}\frac{1}{f(k)}. \eqno(1)$$
Note that $\alpha_f$ may be $+\infty$. Then
$\alpha_f(\mathcal{S})\le\alpha_f$ and
$H_f(\mathcal{S})\subseteq (\inf H_f(\mathcal{S}), \alpha_f(\mathcal{S}))$.
It is clear that $H_f(\mathcal{S})$ is not dense (nowhere dense)
in the interval $(\inf H_f(\mathcal{S}), \alpha_f(\mathcal{S}))$.
However, if we put all the sets $H_f(\mathcal{S})$ together,
then one arrives at the following interesting dense result
that is the second main result of this paper.

\begin{thm}\label{thm2}
Let $f(x)$ be a polynomial of nonnegative integer coefficients and let
$U_f$ be the union set defined by
$$U_f:=\bigcup_{\mathcal{S}\in (\mathbb{Z}^+)^{\infty}}H_f(\mathcal{S}). $$

{\rm (i).} If $\deg f(x)=1$, then $U_f$ is dense in the interval
$(\delta, +\infty)$ with $\delta=1$ if $f(x)=x$, and $\delta=0$
otherwise.

{\rm (ii).} If $\deg f(x)\ge 2$ and
$$\frac{1}{f(k)}\le\sum_{i=1}^\infty\frac{1}{f(k+i)} \eqno(2)$$
holds for all positive integers $k$, then $U_f$ is dense in the
interval $(0, \alpha_f)$ with $\alpha_f$ being given in (1).
\end{thm}

Let $H(\mathcal{S}_n):=H_f(\mathcal{S}_n)$ and
$H(\mathcal{S}):=H_f(\mathcal{S})$ if $f(x)=x^2+1$.
It is well known that (see, for instance, \cite{[KKS]})
$$\sum\limits_{k=1}^{\infty}\frac{1}{k^2+1}=
\frac{1}{2}\Big(\pi \frac{e^{2\pi}+1}{e^{2\pi}-1}-1\Big):=\alpha. \eqno(3)$$
Furthermore, $\alpha\approx 1.076674$. Evidently, for any positive
integer $n$, we have
\begin{align*} \label{eq11}
0<H(\mathcal{S}_n)\le \sum\limits_{k=1}^n\frac{1}{k^2+1}
<\sum\limits_{k=1}^\infty\frac{1}{k^2+1}<2.
\end{align*}
Theorem 1.1 tells us that $H(\mathcal{S}_n)$
is never equal to 1 for any infinite sequence $\mathcal{S}$
of positive integers and any positive integer $n$. This extends
the corresponding result in \cite{[LHQW]} and \cite{[YHYQ]}
which states that for the infinite sequence $\mathcal{S}$
with $s_i= s_1$ for all integers $i \ge 1$, $H(\mathcal{S}_n)$
is not equal to 1. On the other hand, one can easily check
that (2) is true when $f(x)=x^2+1$. So Theorem 1.2
infers that for any sufficiently small $\varepsilon >0$,
there are positive integers $n_1$ and $n_2$ and infinite
sequences $\mathcal{S}^{(1)}$ and $\mathcal{S}^{(2)}$
of positive integers such that
$1-\varepsilon<H(\mathcal{S}^{(1)}_{n_1})<1$ and
$1<H(\mathcal{S}^{(2)}_{n_2})<1+\varepsilon$.

This paper is organized as follows. First, in Section 2, we
recall the early results due to Erd\H{o}s \cite{[E]} and Kakeya
\cite{[Ka]}, respectively, and then show some preliminary lemmas
which are needed in the proofs of Theorems \ref{thm1} and \ref{thm2}.
Then in Sections 3 and 4, we supply the proofs of Theorems \ref{thm1}
and \ref{thm2}, respectively. The final section is devoted to
some remarks. Actually, a conjecture on the case of integer
coefficients polynomial is proposed there.

\section{Auxiliary lemmas}

In this section, we present several auxiliary lemmas
that are needed in the proofs of Theorems {\ref{thm1}}
and {\ref{thm2}}. We begin with a well-known result due to Erd\H{o}s.

\begin{lem}\label{lem2.1} \cite{[E]} For any real number $\xi \ge 7$,
there exists a prime $p\in(\xi, 2\xi]$ such that $p \equiv 1 \pmod 4$.
\end{lem}


For any given prime $p$ with $p \equiv 1 \pmod 4$,
the congruence $x^2+1\equiv 0 \pmod p$ is
solvable, and in the remaining part of this paper,
we use $r_p$ to stand for the smallest positive root of
$x^2+1\equiv 0 \pmod p$. In the conclusion of this section,
we use Lemma \ref{lem2.1} to show the following result that
is vital in the proof of Theorem \ref{thm1}.

\begin{lem}\label{lem2.2} For any integer $n\ge 2$, there is
a prime $p$ with $p\equiv 1\pmod 4$ such that $r_p\le n<p$.
\end{lem}

\begin{proof}
If $n = 2, 3$ or 4, then letting $p = 5$ gives us that $r_p = 2$.
So Lemma \ref{lem2.2} is true in this case.

If $n=5$ or 6, then picking $p = 13$ gives us that $r_p = 5$.
Lemma \ref{lem2.2} holds in this case.

Now let $n \ge 7$. At this moment, Lemma \ref{lem2.1} guarantees
the existence of a prime $p$ such that $p\equiv 1\pmod 4$
and $n < p < 2n$. Since $p-r_p$ is another positive root of
$x^2+1\equiv 0 \pmod p$ and $r_p<p-r_p$, it follows that
$$r_p\le\frac{p-1}{2}<\frac{p}{2}<n<p.$$
as required. Hence Lemma \ref{lem2.2} is proved.
\end{proof}

Now let us state a result obtained by Kakeya in 1914.

\begin{lem}\label{lem2.3} \cite{[Ka]}
Let $\sum_{k=1}^{\infty}a_k$ be an absolutely convergent
infinite series of real numbers and let the set, denoted by $SPS$,
of all the partial sums of the series $\sum_{k=1}^{\infty}a_k$ be
defined by
$$SPS:=\Big\{\sum\limits_{i=1}^{m}a_{k_i} \Big| m\in
\mathbb{Z}^+\cup \{ \infty \}, 1\le k_1<...<k_m \Big\}.$$
Let $u:=\inf SPS$ and $v:=\sup SPS$ (note that $u$ may
be $-\infty$ and $v$ may be $+\infty$). Then the set $U$
consists of all the values in the interval $(u, v)$
if and only if
$$|a_k|\le \sum_{i=1}^\infty|a_{k+i}|$$
holds for all $k\in \mathbb{Z}^+$.
\end{lem}

Using Lemma \ref{lem2.3}, we can prove the following
two useful results that play key roles in the proof
of Theorem \ref{thm2}.

\begin{lem}\label{lem2.4}
Let $\sum_{k=1}^{\infty}a_k$ be a convergent
infinite series of positive real numbers and
$$V:=\Big\{\sum\limits_{i=1}^{m}a_{k_i} \Big| m\in
\mathbb{Z}^+, 1\le k_1<...<k_m \Big\}. $$
If
$$a_k\le \sum_{i=1}^\infty a_{k+i} \eqno(4)$$
holds for all $k\in \mathbb{Z}^+$, then the set $V$
is dense in the interval $(0, v)$ with $v:=\sum_{k=1}^{\infty}a_k$.
\end{lem}

\begin{proof}
From the condition (4) and Lemma \ref{lem2.3}, we know that the set
$$SPS=\Big\{\sum\limits_{i=1}^{m}a_{k_i} \Big| m\in
\mathbb{Z}^+\cup \{ \infty \}, 1\le k_1<...<k_m \Big\}$$
consists of all the values in the interval $(0, v)$ since here $\inf SPS=0$.
Let $r$ be any given real number in $(0, v)$ and $\varepsilon$
be any sufficiently small positive number (one may let
$\varepsilon<\min(r, v-r)$). Then $r\in SPS$ which
implies that there is an integer $m\in\mathbb{Z}^+\cup \{ \infty \}$
and there are $m$ integers $k_1, ..., k_m$ with $1\le k_1<...<k_m$
such that $r=\sum_{i=1}^{m}a_{k_i}$.

If $m\in\mathbb{Z}^+$, then $r\in V$. Lemma 2.4 is true
in this case.

If $m=\infty $, then $r=\sum_{i=1}^{\infty}a_{k_i}$.
That is, ${\rm limit}_{n\rightarrow \infty}\sum_{i=1}^n a_{k_i}=r$.
Thus there is a positive integer $m'$ such that
$|r-\sum_{i=1}^{m'}a_{k_i}|<\varepsilon$.
Noticing that all $a_{k_i}$ are positive,
we deduce that $r-\varepsilon<\sum_{i=1}^{m'}a_{k_i}<r$
as desired.

This completes the proof of Lemma \ref{lem2.4}.
\end{proof}

\begin{lem}\label{lem2.5}
Let $\sum_{k=1}^{\infty}a_k$ be a divergent infinite series
of positive real numbers with $a_k$ decreasing as $k$ increasing
and $a_k\rightarrow 0$ as $k\rightarrow\infty$. Define
$$V:=\Big\{\sum\limits_{i=1}^{m}a_{k_i}\Big| m\in
\mathbb{Z}^+, 1\le k_1<...<k_m \Big\}. $$
Then the set $V$ is dense in the interval $(0, +\infty)$.
\end{lem}

\begin{proof}
Let $r$ be any given real number in $(0, +\infty)$ and $\varepsilon$
be any sufficiently small positive number (one may let $\varepsilon<r$).
Let $a_0:=0$ and $m_0=0$. Since the series $\sum_{k=0}^{\infty}a_k$ is
divergent, there exists a unique integer $m_1\ge 0$ such that
$$\sum\limits_{k=m_0}^{m_1}a_k<r$$
and
$$\sum\limits_{k=m_0}^{m_1}a_k+a_{m_1+1}\ge r. $$

On the one hand, since $a_k$ decreases as $k$ increases
and $a_k\rightarrow 0$ as $k\rightarrow\infty$,
there is an integer $m_2$ with $m_2>m_1+1$ and
$$a_{m_2}<r-\sum\limits_{k=m_0}^{m_1}a_k\le a_{m_1+1}. $$
Moreover, there exists an integer $m_3$ with $m_3\ge m_2$ and
$$\sum\limits_{k=m_0}^{m_1}a_k+\sum\limits_{k=m_2}^{m_3}a_k<r$$
and
$$\sum\limits_{k=m_0}^{m_1}a_k+\sum\limits_{k=m_2}^{m_3}a_k+a_{m_3+1}\ge r$$
since $\sum_{k=m_2}^{\infty}a_k$ also diverges.

Continuing in this way, we can form an increasing sequence
$\{m_k\}_{k=0}^{\infty}$ such that
$$\sum\limits_{k=m_0}^{m_1}a_k+\sum\limits_{k=m_2}^{m_3}a_k+\dots
+\sum\limits_{k=m_{2t}}^{m_{2t+1}}a_k<r$$
but
$$\sum\limits_{k=m_0}^{m_1}a_k+\sum\limits_{k=m_2}^{m_3}a_k+\dots
+\sum\limits_{k=m_{2t}}^{m_{2t+1}}a_k+a_{m_{2t+1}+1}\ge r$$
for any nonnegative integer $t$. Obviously, one has
$$\sum_{k=m_0}^{m_1}a_k+\sum\limits_{k=m_2}^{m_3}a_k
+\dots+\sum_{k=m_{2t}}^{m_{2t+1}}a_k\in V.$$

On the other hand, since $\lim_{k\rightarrow+\infty}a_k=0$, it follows that
there exists a nonnegative integer $t_0$ such that $a_{m_{2t_0+1}+1}<\varepsilon$.
That is, we have
$$r-\varepsilon<r-a_{m_{2t_0+1}+1}\le \sum\limits_{k=m_0}^{m_1}a_k
+\sum\limits_{k=m_2}^{m_3}a_k+\dots+\sum\limits_{k=m_{2t_0}}^{m_{2t_0+1}}a_k<r.$$
Hence $V$ is dense in the interval $(0, +\infty)$.

This concludes the proof of Lemma \ref{lem2.5}.
\end{proof}

\section{Proof of Theorem \ref{thm1}}
As usual, for any prime $p$ and for any integer $x$,
we let $v_p(x)$ stand for the {\it $p$-adic valuation}
of $x$, i.e., $v_p(x)$ is the biggest nonnegative integer
$r$ with $p^r$ dividing $x$. If $x=\frac{a}{b}$,
where $a$ and $b$ are integers and $b\ne 0$,
then define $v_p(x):=v_p(a)-v_p(b)$.

We can now prove Theorem \ref{thm1} as follows.\\

{\it Proof of Theorem \ref{thm1}.}
We just need to prove that $H_{f}(\mathcal{S}_n)$ is between two adjacent integers
or $v_p(H_f(\mathcal{S}_n))<0$ for some prime $p$.
Let $f(x)=a_mx^m+a_{m-1}x^{m-1}+\dots +a_1x+a_0$ with $m\ge 2$ and $a_m\ge 1$.

If there is some $a_i\ne 0$ with $1\le i\le m-1$, then
$f(k)\ge a_mk^m+a_ik^i\ge k^2+k$ for all $k\ge 1$. Therefore,
$$H_{f}(\mathcal{S}_n)=\sum\limits_{k=1}^{n}\frac{1}{f(k)^{s_{k}}}
\le \sum\limits_{k=1}^{n}\frac{1}{k^2+k}
=\sum\limits_{k=1}^{n}\Big(\frac{1}{k}-\frac{1}{k+1}\Big)
=1-\frac{1}{n+1}<1. $$

If $a_i=0$ for all $0\le i\le m-1$, then $f(x)=a_mx^m$.
Furthermore,
$$1<H_{f}(\mathcal{S}_n)=\sum\limits_{k=1}^{n}\frac{1}{(k^{m})^{s_k}}
\le \sum\limits_{k=1}^{n}\frac{1}{k^2}<\zeta(2)=\frac{\pi^2}{6}<2$$
when $a_m=1$, and
$$0<H_{f}(\mathcal{S}_n)=\sum\limits_{k=1}^{n}\frac{1}{(a_mk^{m})^{s_k}}
\le \frac{1}{2}\sum\limits_{k=1}^{n}\frac{1}{k^2}<1. $$
when $a_m\ge 2$.

If $a_i=0$ for all $1\le i\le m-1$ and $a_0\ne 0$,
then $f(x)=a_mx^m+a_0$. Moreover, if $a_m\ge 2$ or $a_0\ge 2$,
then $f(k)=k^m+(a_m-1)k^m+a_0\ge k^2+2$ for all $k\ge 1$. So
$$0<H_{f}(\mathcal{S}_n)\le \sum\limits_{k=1}^{n}\frac{1}{k^2+2}
\le \sum\limits_{k=1}^{n}\frac{1}{k^2+1}-\frac{1}{2}-\frac{1}{5}+
\frac{1}{3}+\frac{1}{6}< \sum\limits_{k=1}^{{\infty}}\frac{1}{k^2+1}
-\frac{1}{5}<1. $$
If $m\ge 3$, then $f(k)\ge k^3+1$ for all $k\ge 1$. So
$$0<H_{f}(\mathcal{S}_n)\le \sum\limits_{k=1}^{n}\frac{1}{k^3+1}
\le \sum\limits_{k=1}^{n}\frac{1}{k^3}-1-\frac{1}{8}+
\frac{1}{2}+\frac{1}{9}< \zeta (3)-\frac{37}{72}<1. $$

In what follows, we let $f(x)=x^2+1$.

By Lemma \ref{lem2.2},
there is a prime $p$ such that $p \equiv 1 \pmod 4$ and $r_p\le n<p$
where $r_p$ is the smallest positive root of $x^2+1\equiv 0 \pmod p$.
Since $p\mid(r_p^2+1)$, one has $r_p^2\ge p-1$.
Noticing that $p\ge 5$, it follows that
$$2 \le \sqrt{p-1} \le r_p \le \frac{p-1}{2}<p-r_p.$$
Therefore
$$0<r_p^2+1<(p-r_p)^2+1\le (p-2)^2+1<p^2.$$
This infers that $v_p(r_p^2+1)=1$ and $v_p((p-r_p)^2+1)=1$.
So we have
$$v_p\Big(\dfrac{1}{(r_p^2+1)^{s_{r_p}}}\Big)=-s_{r_p}<0,$$
$$v_p\Big(\dfrac{1}{{((p-r_p)^2+1)}^{s_{p-r_p}}}\Big)=-s_{p-r_p}<0$$
and
$$v_p\Big(\dfrac{1}{(k^2+1)^{s_{k}}}\Big)=0$$
for any integer $k$ with $1\le k\le p$ and $k\not\in\{r_p, p-r_p\}$.

Now we divide the proof into the following two cases.

{\sc Case 1.} $r_p \le n < p-r_p$. Since
$$v_p\Big(\sum_{k=1\atop k\ne r_p}^n \dfrac{1}{(k^2+1)^{s_{k}}}\Big)\ge 0,$$
it follows from the isosceles triangle principle (see, for example, \cite{[K]})
that
$$v_p(H(\mathcal{S}_n))=v_p\Big(\dfrac{1}{(r_p^2+1)^{s_{r_p}}}
+\sum_{k=1\atop k\ne r_p}^n \dfrac{1}{(k^2+1)^{s_{k}}}\Big)
=v_p\Big(\dfrac{1}{(r_p^2+1)^{s_{r_p}}}\Big)=-s_{r_p}<0.$$
Namely, $H(\mathcal{S}_n) \notin \mathbb{Z}$.
So Theorem 1.1 is proved in this case.

{\sc Case 2.} $p-r_p \le n < p$. Let $H(\mathcal{S}_n)=A+B$, where
$$A:=\frac{1}{(r_p^2+1)^{s_{r_p}}}+\frac{1}{((p-r_p)^2+1)^{s_{p-r_p}}}$$
and
$$B:=\sum_{k=1\atop k\ne r_p, k\ne p-r_p}^n\dfrac{1}{(k^2+1)^{s_{k}}}.
$$
Evidently, one has $v_p(B)\ge 0$.
We claim that $v_p(A)<0$. Then by the claim and
the isosceles triangle principle again, we obtain that
\begin{align*}
v_p(H(\mathcal{S}_n))&=v_p(A+B)=v_p(A)<0,
\end{align*}
which implies that $H(\mathcal{S}_n) \notin \mathbb{Z}$ as desired.
It remains to show the truth of the claim.

If $s_{r_p}\ne s_{p-r_p}$, then it is obvious that
$v_p(A)=\min(-s_{r_p},-s_{p-r_p})<0.$
So the claim is true if $s_{r_p}\ne s_{p-r_p}$.

Now let $s_{r_p} = s_{p-r_p}:= s$. Then
$$A=\frac{((p-r_p)^2+1)^s+(r_p^2+1)^s}{(r_p^2+1)^s((p-r_p)^2+1)^s}.$$
We introduce an auxiliary function $g(x)$ as follows:
$$g(x):=\Big(\Big(\frac{p}{2}+x\Big)^2+1\Big)^s
+\Big(\Big(\frac{p}{2}-x\Big)^2+1\Big)^s.$$
Then the derivative of $f(x)$ is
$$g'(x)=s\Big(\Big(\frac{p}{2}+x\Big)^2+1\Big)^{s-1}(p+2x)-
s\Big(\Big(\frac{p}{2}-x\Big)^2+1\Big)^{s-1}(p-2x).$$
So $g'(x)>0$ if $x \in (0,\frac{p}{2}]$.
This implies that $g(x)$ is increasing if $x\in (0,\frac{p}{2}]$.

Since $2\le r_p<\frac{p}{2}$, one derives that
$0<\frac{p}{2}-r_p\le \frac{p}{2}-2<\frac{p}{2}$.
Hence
\begin{align*}
((p-r_p)^2+1)^s+(r_p^2+1)^s&=g\Big(\frac{p}{2}-r_p\Big)\\
&\le g\Big(\frac{p}{2}-2\Big) \\
&=(2^2+1)^s+((p-2)^2+1)^s\\
&< 5^s+(p-1)^{2s}\\
&< p^{2s},
\end{align*}
where the last inequality follows from the fact
that $p\ge 5$ implies that
$$p^{2s}-(p-1)^{2s}=(p^s+(p-1)^s)(p^s-(p-1)^s)>5^s.$$
Thus
$$v_p\big(((p-r_p)^2+1)^s+(r_p^2+1)^s\big)<2s.$$
Therefore
\begin{align*}
v_p(A) &= v_p\big(((p-r_p)^2+1)^s+(r_p^2+1)^s\big)
-v_p\big((r_p^2+1)^s((p-r_p)^2+1)^s\big)\\
&< 2s-2s = 0.
\end{align*}
The claim holds if $s_{r_p}=s_{p-r_p}$. The claim is proved.

This finishes the proof of Theorem \ref{thm1}.
\hfill$\Box$

\section{Proof of Theorem \ref{thm2}}

In the section, we present the proof of Theorem \ref{thm2}.\\

{\it Proof of Theorem \ref{thm2}.} Let
$$V_f:=\Big\{\sum\limits_{i=1}^{m}\frac{1}{f(k_i)}\Big| m\in
\mathbb{Z}^+, 1\le k_1<...<k_m \Big\}$$
and
$$\bar V_f:=\Big\{\sum\limits_{i=1}^{m}\frac{1}{f(k_i)}\Big| m\in
\mathbb{Z}^+, 2\le k_1<...<k_m\Big\}. $$
Pick any given real number $r$ in $(\inf U_f, \sup U_f)$ and let
$\varepsilon$ be any sufficiently small positive number
(one may let $\varepsilon<\min(r-\inf U_f, \sup U_f-r )$).

(i). Since $f(x)$ is a polynomial of nonnegative integer and
degree one, it follows that $\sum_{k=1}^{\infty}\frac{1}{f(k)}$
(resp. $\sum_{k=2}^{\infty}\frac{1}{f(k)}$)
is a divergent infinite series of positive real numbers with
$\big\{\frac{1}{f(k)}\big\}_{k=1}^{\infty}$
(resp. $\big\{\frac{1}{f(k)}\big\}_{k=2}^{\infty}$)
directly decreasing to $0$ as $k$ increases. By Lemma \ref{lem2.5},
we know that $V_f$ (resp. $\bar V_f$) is dense in the interval
$(0, +\infty )$. Clearly, we have $\sup U_f=\sup V_f=+\infty$.

If $f(1)=1$, then $f(x)=x$ which implies that $f(2)>1$,
$\inf U_f=1$ and $r\in (\inf U_f, \sup U_f)=(1, +\infty)$.
Since $\bar V_f$ is dense in the interval $(0, +\infty )$,
there is an element
$$\sum\limits_{i=1}^{m}\frac{1}{f(k_i)}\in
\Big(r-1-\varepsilon, r-1-\frac{\varepsilon}{2}\Big)\eqno(5)$$
with $2\le k_1<\dots<k_m$. Now let $s_k=1$ for
$k\in \{k_1, \dots ,k_m\}$ and
$s_k>\frac{\log\frac{2k_m}{\varepsilon}}{\log f(2)}$
for $k\in \{2, 3, \dots ,k_m\}\setminus \{k_1, \dots ,k_m\}$.
Then
$$0\le \sum\limits_{k=2\atop k\not\in
\{k_1, \dots ,k_m\}}^{k_m}\frac{1}{f(k)^{s_k}}
<\frac{k_m}{f(2)^{\frac{\log\frac{2k_m}{\varepsilon}}{\log f(2)}}}
=\frac{\varepsilon}{2}. \eqno(6)$$
It follows from (5) and (6) that
$$\sum\limits_{k=1}^{k_m}\frac{1}{f(k)^{s_k}}=1+
\sum\limits_{k=2\atop k\not\in \{k_1, \dots ,k_m\}}^{k_m}\frac{1}{f(k)^{s_k}}
+\sum\limits_{i=1}^{m}\frac{1}{f(k_i)^{s_{k_i}}}
\in (r-\varepsilon, r). $$
That is, $U_f$ is dense in the interval $(\inf U_f, \sup U_f)=(1, +\infty)$
in this case.

If $f(1)>1$, then $\inf U_f=0$ and $r\in (\inf U_f, \sup U_f)=(0, +\infty )$.
Since $V_f$ is dense in the interval $(0, +\infty )$,
there is an element
$$\sum\limits_{i=1}^{m}\frac{1}{f(k_i)}\in
\Big(r-\varepsilon, r-\frac{\varepsilon}{2}\Big) \eqno(7)$$
with $1\le k_1<\dots<k_m$. Now, let $s_k=1$ for
$k\in \{k_1, \dots ,k_m\}$ and
$s_k>\frac{\log\frac{2k_m}{\varepsilon}}{\log f(1)}$
for $k\in \{1, 2, \dots ,k_m\}\setminus \{k_1, \dots ,k_m\}$.
One has
$$0\le \sum\limits_{k=1\atop k\not\in
\{k_1, \dots ,k_m\}}^{k_m}\frac{1}{f(k)^{s_k}}
<\frac{k_m}{f(1)^{\frac{\log\frac{2k_m}{\varepsilon}}{\log f(1)}}}
=\frac{\varepsilon}{2},  \eqno(8)$$
and so by (7) and (8),
$$\sum\limits_{k=1}^{k_m}\frac{1}{f(k)^{s_k}}=
\sum\limits_{k=1\atop k\not\in \{k_1, \dots ,k_m\}}^{k_m}\frac{1}{f(k)^{s_k}}
+\sum\limits_{i=1}^{m}\frac{1}{f(k_i)^{s_{k_i}}}
\in (r-\varepsilon, r). $$
Namely, $U_f$ is dense in the interval $(\inf U_f, \sup U_f)=(0, +\infty )$
in this case.

(ii). First of all, since $f(x)$ is a polynomial of nonnegative
integer and $\deg f(x)\ge 2$, we know that $\sum_{k=1}^{\infty}\frac{1}{f(k)}$
is a convergent infinite series of positive real numbers.
With the hypothesis $\frac{1}{f(k)}\le\sum_{i=1}^\infty\frac{1}{f(k+i)}$
for any positive integer $k$, Lemma \ref{lem2.4} yields that $V_f$
is dense in the interval $(0, \sup V_f)$.

We claim that $f(1)>1$. Otherwise, $f(1)=1$.
Then $f(x)=x^m$ with $m\ge 2$. However,
$$\frac{1}{f(1)}=1>\frac{\pi ^2}{6}-1
=\sum_{i=1}^\infty\frac{1}{(1+i)^2}
\ge \sum_{i=1}^\infty\frac{1}{f(1+i)}, $$
which contradicts with our hypothesis. So we
must have $f(1)>1$. The claim is proved.

In the following, we let $f(1)>1$. Then $\inf U_f=0$,
$\sup U_f=\sup V_f=\alpha_f$
and $r\in (\inf U_f, \sup U_f)=(0, \alpha_f)$.
Since $V_f$ is dense in the interval $(0, \sup V_f)
=(0, \alpha_f)$, there is an element
$$\sum\limits_{i=1}^{m}\frac{1}{f(k_i)}\in
\Big(r-\varepsilon, r-\frac{\varepsilon}{2}\Big)$$
with $1\le k_1<\dots<k_m$. Then letting $s_k=1$ for
$k\in \{k_1, \dots ,k_m\}$ and
$s_k>\frac{\log\frac{2k_m}{\varepsilon}}{\log f(1)}$
for $k\in \{1, 2, \dots ,k_m\}\setminus \{k_1, \dots ,k_m\}$
gives us that
$$0\le \sum\limits_{k=1\atop k\not\in
\{k_1, \dots ,k_m\}}^{k_m}\frac{1}{f(k)^{s_k}}
<\frac{k_m}{f(1)^{\frac{\log\frac{2k_m}{\varepsilon}}{\log f(1)}}}
=\frac{\varepsilon}{2}.$$
It infers that
$$\sum\limits_{k=1}^{k_m}\frac{1}{f(k)^{s_k}}=
\sum\limits_{k=1\atop k\not\in \{k_1, \dots ,k_m\}}^{k_m}
\frac{1}{f(k)^{s_k}}+\sum\limits_{i=1}^{m}\frac{1}{f(k_i)^{s_{k_i}}}
\in (r-\varepsilon, r). $$
In other words, $U_f$ is dense in the interval
$(0, \alpha_f)$. So part (ii) is proved.

The proof of Theorem \ref{thm2} is complete.
\hfill$\Box$

\section{Concluding remarks}
1. Let $\mathcal{T}:=((0, \alpha)\bigcap\mathbb{Q})
\setminus \bigcup\limits_{\mathcal{S}\in (\mathbb{Z}^+)
^{\infty}}H(\mathcal{S})$. Then Theorem 1.1 tells us that
$1\in \mathcal{T}$. But it is well known that if $p$ is
a prime, then the congruence $x^2+1\equiv 0\pmod p$ is
solvable if and only if either $p=2$, or $p\equiv 1\pmod 4$.
Thus for any infinite sequence $\mathcal{S}$ of positive
integers and for any positive integer $n$, if one writes
$H(\mathcal{S}_n)=\frac{H_1(\mathcal{S}_n)}{H_2(\mathcal{S}_n)}$,
where $H_1(\mathcal{S}_n), H_2(\mathcal{S}_n)\in\mathbb{Z}^+$
and $\gcd(H_1(\mathcal{S}_n), H_2(\mathcal{S}_n))=1$,
then $H_2(\mathcal{S}_n)$ is not divisible by any prime
$p$ with $p\equiv 3\pmod 4$. It follows that
$\mathcal{L}:=\{ \frac{a}{b}\in (0, \alpha)
|a, b\in\mathbb{Z}, (a, b)=1, b\ {\rm is \ divisible \ by
\ at \ least \ a \ prime} \ p \ {\rm with} \ p\equiv 3\pmod 4\}
\subset\mathcal{T}$.
An interesting question naturally arises:
Are there other elements in the set $\mathcal{T}$
except for the elements in $\{1\}\bigcup\mathcal{L}$?
Further, one would like to determine the set
$\mathcal{T}$. This problem is kept open so far.

2. We let $f(x)$ be a polynomial of nonnegative integer coefficients
and of degree at least two, and let $U_f$ be the union set given
in Theorem \ref{thm2}. Then part (ii) of Theorem \ref{thm2}
says that the condition (2) is a sufficient condition such that
the union set $U_f$ is dense in the interval $(0, \alpha_f)$.
One may ask the following interesting question: What is the
sufficient and necessary condition for the union set $U_f$
to be dense in the interval $(0, \alpha_f)$?

3. Now let $f(x)$ be a nonzero polynomial of integer coefficients. Let
$Z_f:=\{x\in \mathbb{Z}: f(x)=0\}$ be the set of integer roots of $f(x)$
and $\{a_i\}_{i=1}^{\infty}:=\mathbb{Z}^+\setminus Z_f$ be arranged in the
increasing order. Then $f(a_i)\ne 0$ for all integers $i\ge 1$.
Let $n$ and $k$ be integers such that $1\le k\le n$ and let
$H_f^{(k)}(\mathcal{S}_n)$ stand for the $k$-th elementary symmetric
functions of
$$\frac{1}{f(a_1)^{s_1}}, \frac{1}{f(a_2)^{s_2}}, ..., \frac{1}{f(a_n)^{s_n}}.$$
That is,
$$H_f^{(k)}(\mathcal{S}_n):=\sum\limits_{1\leq i_{1}<\cdots<i_{k}\le n}
\prod\limits_{j=1}^{k}\frac{1}{f(a_{i_{j}})^{s_{i_j}}}.$$
Then $H_f^{(1)}(\mathcal{S}_n)=H_f(\mathcal{S}_n)$. Let
$$\bar H_f^{(k)}(\mathcal{S}_n):=\sum\limits_{1\leq i_{1}\le\cdots\le i_{k}\le n}
\prod\limits_{j=1}^{k}\frac{1}{f(a_{i_{j}})^{s_{i_j}}}.$$
When $f(x)$ is of nonnegative integer coefficients and $s_i=1$ for
all integers $i\ge 1$, the integrality of $H_f^{(k)}(\mathcal{S}_n) (1\le k\le n)$
was previously investigated in \cite{[CT]}, \cite{[EN]},
\cite{[HW]}, \cite{[LHQW]} and \cite{[WH]}. But such integrality
problem has not been studied when $f(x)$ contains negative coefficients.
On the one hand, for any given integer $N_0\ge 1$, one can easily
find a polynomial $f_0(x)$ of integer coefficients such that
for all integers $n$ and $k$ with $1\le k\le n\le N_0$, both of
$H_{f_0}^{(k)}(\mathcal{S}_n)$ and $\bar H_{f_0}^{(k)}(\mathcal{S}_n)$
are integers. Actually, letting
$$f_0(x)=\prod_{i=1}^{N_0}(x-i)\pm 1$$
gives us the expected result. On the other hand, for any given nonzero
polynomial $f(x)$ of
integer coefficients, we believe that the similar integrality result
is still true. So in concluding this paper, we suggest the following
more general conjecture that generalizes Conjecture 4.1 of \cite{[FHJY]}
and Conjecture 3.1 of \cite{[LHQW]}.

\begin{con}
Let $f(x)$ be a nonzero polynomial of integer coefficients and
$\mathcal{S}=\{s_i\}_{i=1}^\infty$ be an infinite sequence of
positive integers (not necessarily increasing and not necessarily
distinct). Then there is a positive integer $N$ such that for
any integer $n\ge N$ and for all integers $k$ with $1\le k\le n$,
both of $H_f^{(k)}(\mathcal{S}_n)$ and $\bar H_f^{(k)}(\mathcal{S}_n)$
are not integers.
\end{con}

By Theorem 1.1, one knows that Conjecture 5.1 holds when
$k=1$. It is clear that Conjecture 5.1 is true when $k=n$.
Thus we need just to look at the case $2\le k\le n-1$.
Obviously, the results presented in \cite{[CT]},
\cite{[EN]}-\cite{[HW]}, \cite{[LHQW]}-\cite{[YLFJ]}
and Theorem 1.1 of this paper supply
evidences to Conjecture 5.1.

\bibliographystyle{amsplain}

\end{document}